\documentclass[a4paper,reqno]{amsart}
\usepackage{amsmath,amsfonts,amssymb,enumerate,amsthm,mathrsfs,graphics,color,graphicx,datetime}
\usepackage{newtxtext,newtxmath}  
\usepackage{dsfont}

\usepackage[colorlinks=true]{hyperref} 

\theoremstyle{plain}
\newtheorem{thm}{Theorem}[section]
\newtheorem{lem}[thm]{Lemma}

\theoremstyle{definition}
\newtheorem{defi}[thm]{Definition}
\theoremstyle{remark}
\newtheorem{rem}[thm]{Remark}

\def\Rde{{\mathbb{R}^d}}
\def\Ne{{\mathbb{N}}}
\def\Re{{\mathbb R}}

\def\al{\alpha}
\def\ep{\varepsilon}

\def\be{\beta}
\def\dl{\delta}

\def\f{\varphi}
\def\Om{\Omega}

\def\dd{\text{\rm\,d}} 
\def\spt{\operatorname{supp}}  
\def\sT{\,;\;} 

\def\FT{\mathscr F}

\numberwithin{equation}{section}
\title[Quantitative Non-Compactness Properties of the Fourier Transform]
{Quantitative Non-Compactness Properties of the Fourier Transform on Optimal Spaces}

\author[D.E. Edmunds]{David E. Edmunds}
\address{D.E. Edmunds, Department of Mathematics, University of Sussex, Pevensey 2 Building,
	Brighton BN1 9OH, Sussex, United Kingdom}
\email{davideedmunds@aol.com}
\urladdr{\href{https://orcid.org/0000-0002-8528-3067}{0000-0002-8528-3067}}

\author[P. Gurka]{Petr Gurka}
\address{P. Gurka, Department of Mathematics,
Czech University of Life Sciences Prague, 165 21, Prague 6, Czech Republic;
Department of Mathematics, College of Polytechnics Jihlava, Tolstého 16, 586 01, Jihlava, Czech Republic;
Department of Mathematical Analysis, Faculty of Mathematics and Physics, Charles University,
Sokolovská 83, 186 75 Praha 8, Czech Republic}
\email{gurka@tf.czu.cz}
\urladdr{\href{https://orcid.org/0000-0002-0995-4711}{0000-0002-0995-4711}}

\author[J. Lang]{Jan Lang}
\address{J. Lang, Department of Mathematics, The Ohio State University, 231 West 18th Avenue, Columbus,
OH 43210-1174, USA \\ and  Czech Technical University in Prague, Faculty of Electrical Engineering,
Department of Mathematics, Technick\'a~2, 166~27 Praha~6, Czech Republic}
\email{lang.162@osu.edu}
\urladdr{\href{https://orcid.org/0000-0003-1582-7273}{0000-0003-1582-7273}}

\thanks{The research of the second author was supported by grant No.~GA23-04720S
 of the Czech Science Foundation.}

\begin{document}

\begin{abstract}
We establish that the Fourier transform $\FT: L^p(\Rde)\to L^{p',p}(\Rde)$, for $d\in\Ne$ and $1<p<2$,
is not strictly singular, thereby confirming the optimality of the source and target spaces.
A~similar result is obtained for Fourier series on $L^p(\mathbb{T}^n)$,
with sequence Lorentz spaces as the target. These findings complement known results, which state
that $\FT: L^p(\Rde)\to L^{p'}(\Rde)$ is finitely strictly singular and then also strictly singular,
and provide further insight into the degrees of non-compactness of~$\FT$.
\end{abstract}

\maketitle

\section{Introduction}

The Fourier transform is a cornerstone of modern analysis, and a powerful tool.  Despite its extensive use,
the nature of the Fourier transform, especially in the context of Banach spaces, poses several intricate challenges.
	
Understanding the boundedness and compactness properties of the Fourier transform when mapped between function spaces
has been a longstanding area of interest. Previous studies have explored its behavior on rearrangement-invariant
Banach spaces \cite{BoSo18}, \cite{BrCo96}
as well as its compactness between different Besov spaces
\cite{MR4507897}, \cite{MR4571614}.
While compactness is a desirable property, it is often difficult to establish (if it holds at all!).
This limitation has motivated the study of weaker compactness notions such as weak compactness,
complete continuity, and strict singularity.
	
A recent study \cite{LaMi23} highlights that traditional measures of non-compactness are often insufficient to capture
the nuanced levels of non-compactness of linear maps. In such contexts, the concepts of strict singularity provide a
more refined framework, offering valuable insights into the operator's behavior by distinguishing
between different degrees of non-compactness.
Let us recall that an operator is strictly singular if it fails to be an isomorphism
on any infinite-dimensional subspace of the domain (see~\cite{MR3167479} and \cite{MR2098887}).
	
	In this paper, we examine the strict singularity of the Fourier transform $\FT: L^p(\mathbb{R}^d) \
\rightarrow L^{p',p}(\mathbb{R}^d)$ for $1 < p < 2$.
Note that for $p=2$, the Fourier transform is an isometry
(or equivalent to the isometry, depending on the definition of $\FT$), and hence it is not strictly singular.
For $p=1$, the Fourier transform is strictly singular,
though not finitely strictly singular when
underlying domain is compact (see~\cite[Corollary~6]{MR2098887}
and also the paper~\cite{DiDi79}).
We show that, although the Fourier transform from $L^p$ to $L^{p'}$
is finitely strictly singular (see \cite{MR3167479}), it is not strictly singular when the target space is $L^{p',p}$.
This result reconfirm the optimality of the source and target spaces for the Fourier transform and provides new insights
into the quantitative aspects of its non-compactness and complete the previous results.
We obtain a similar result for the discrete Fourier transform into sequence Lorentz spaces.

In the next section, we provide the notations, definitions, and preliminary lemmas.
The main theorem (Theorem \ref{MainThhm}) is stated in Section~\ref{Sect3}
and proved in Section~\ref{Sect4}. Section~\ref{Sect5} presents results for the discrete
Fourier transform (Theorem \ref{MainDiscreteThm}). The paper concludes with final remarks in Section~\ref{Sect6}.

\section{Notation and preliminaries}
	
\subsection*{Strict Singularity}

	Let $X$ and $Y$ be Banach spaces, and denote by $B(X, Y)$ the space of all bounded linear operators from $X$ to $Y$.

\begin{defi} \label{DefStrSing}
An operator $T \in B(X, Y)$ is said to be:
	\begin{itemize}
		\item \textit{Finitely strictly singular} if, for every $\ep > 0$, there exists an integer $N_\ep$
such that for any subspace $E \subset X$ with dimension greater than $N_\ep$, there exists $x$ in the unit
sphere of $E$ such that ${\|T(x)\|_Y \leq \ep}$.
		\item \textit{Strictly singular} if it does not induce an isomorphism on any infinite-dimensional closed
subspace of $X$. Specifically, for any infinite-dimensional closed subspace $E \subset X$ and any $\ep > 0$,
there exists $x$ in the unit sphere of $E$ such that ${\|T(x)\|_Y \leq \ep}$.
	\end{itemize}
\end{defi}

	It is well-known that every compact operator is strictly singular, but the converse is not true. This hierarchy can be formally expressed as:
	\[
	\text{compact} \implies \text{finitely strictly singular} \implies \text{strictly singular}.
	\]

\subsection*{Bernstein Numbers and Finite Strict Singularity}
	
	The finite strict singularity of an operator can be characterized using Bernstein numbers.
For $T \in B(X, Y)$, the $k$-th Bernstein number, $b_k(T)$, is defined as:
\begin{equation} \label{FSSaSS}
	b_k(T) = \sup_{X_k} \inf_{x \in X_k \setminus \{0\}} \frac{\|T x\|_Y}{\|x\|_X},
\end{equation}
where the supremum is taken over all $k$-dimensional subspaces $X_k$ of $X$.
	
Obviously $\|T\|=b_1(T)\ge b_2(T)\ge\dots\ge0$.
	 And it is quite easy to see that an operator $T$ is finitely strictly singular if and only if:
\begin{equation*}
  	\lim_{k \to \infty} b_k(T) = 0.
\end{equation*}
	These notions are crucial for understanding the nuanced behavior of the Fourier transform in non-compact
settings, as it allows us to quantify the degree to which the transform fails to be compact.
	
\subsection*{Function Spaces and Preliminary results}

In the next we assume that $\mu$ denotes the $d$-dimensional Lebesgue measure
on~$\Rde$, $d\in\Ne$.
The Lebesgue measurable sets, or functions, we simply call measurable sets, or functions, respectively.

\subsubsection*{Lebesgue space}
We denote by $L^p(\Rde)$, $p\in[1,\infty]$, the \emph{Lebesgue space}
of all measurable functions~$f$
on $\Rde$ which are equal a.e., equipped with the norm
\begin{equation*}
  \|f\|_{p}=
  \bigg\{
    \begin{array}{ll} \vspace{3pt}
      \big(\int_\Om |f(x)|^p\dd x\big)^{1/p} & \hbox{if $p<\infty$,} \\
      \text{ess sup\,}_{x\in\Rde}|f(x)| & \hbox{if $p=\infty$.}
    \end{array}
\end{equation*}
For $p\in[1,\infty]$ we define the \emph{H\"older conjugate}
number $p'\in[1,\infty]$ by the equality
\begin{equation*}
  \frac1p+\frac1{p'}=1.
\end{equation*}
 Given a measurable function $f:$ $\mathbb{R}%
^{n}\rightarrow \mathbb{R}$ we introduce its \textit{distribution function }$%
f_{\ast }$ by
\begin{equation}\label{DistF}
  f_{*}(\tau)=\mu\{x\in \Rde\sT |f(x)|>\tau\}, \quad \tau>0,
\end{equation}
and its \emph{nonincreasing rearrangement}
\begin{equation} \label{NonRear}
  f^*(t)=\inf\{\tau>0\sT f_{*}(\tau)\le t\},\quad t>0.
\end{equation}

\subsubsection*{Lorentz space}
Assume that $f$ is a measurable real-valued function
on  $\Rde$ and let $r,s\in(0,\infty)$.
Define
\begin{equation} \label{LorqN}
	\|f\|_{r,s}=
	\bigg\{
	\begin{array}{ll} \vspace{4pt}
		\big(\int_0^\infty(t^{1/r}f^*(t))^s\frac{\dd t}{t}\big)^{1/s} & \hbox{if $s<\infty$,} \\
		\sup_{t>0}\,t^{1/r}f^*(t) & \hbox{if $s=\infty$.}
	\end{array}
\end{equation}
The set of all such $f$ satisfying $\|f\|_{r,s}<\infty$, denoted by $L^{r,s}(\Rde)$, is called
the \emph{Lorentz space} with indices~$r, s$.

\begin{rem}
	It is easy to see that, if $\Psi$ is a continuous and increasing function, then
	\begin{equation*}
		\int_\Rde\Psi\big(|f(x)|\big)\dd x
		=\int_{0}^{\infty}\Psi\big(f^*(t)\big)\dd t.
	\end{equation*}
	In particular, it implies that, for $p\in[1,\infty)$,
	$\|f\|_{p}=\|f\|_{p,p}$.
	Thus,
	$L^p(\Rde)=L^{p,p}(\Rde)$.
\end{rem}

\subsubsection*{Fourier transform}
For a function $f\in L^1(\Rde)$ we define its \emph{Fourier transform}
\begin{equation}\label{FourTrans}
	\FT(f)(x)
	=
	\widehat{f}(x)
	=
	\int_{\Rde}f(t)e^{-i\langle x,t\rangle}\dd t,
\end{equation}
where, for $x=(x_1,\dots,x_d)$, $t=(t_1,\dots,t_d)$,
$\langle x,t\rangle=\sum_{k=1}^{d}x_kt_k$.
\newline
We recall some basic properties of the Fourier transform:
\begin{enumerate}[{\rm(i)}]
	\item\label{lin}
	{Linearity} \quad
	$\FT\big(\al f+\be g\big)=\al\,\FT(f)+\be\,\FT(g)$\quad ($\al$, $\be$ constants);
	\item \label{psVz}
	{Scaling}\quad
	$\FT\big(f(c\,t)\big)(x)=c^{-d}\,\widehat{f}\big(x/c\big)$
	\quad($c>0$ constant).
\end{enumerate}
Obviously
\[  |\FT\big(f\big)(x)| \le \int_{\Rde}  |f(x)| dx \]
and  by Parseval's formula we have
\[\int_{\Rde} |\FT\big(f\big)(x)|^2 dx= (2\pi)^d \int_{\Rde}  |f(x)|^2 dx.\]
Combining this with Riesz-Thorin interpolation theorem we get
\[\FT: L^p(\Rde) \to L^{p'}(\Rde), \mbox{ for } 1\le p \le 2,
\]
(for more details see \cite[Section 1.2]{MR482275}).

The following result, the proof of which is straightforward, will be needed later.
\begin{lem} \label{DilatL}
Let $F: \Rde\to\Re$ be a measurable function. For $c>0$ denote
\begin{equation*}
  F_c(x)=F(c\,x), \quad x\in\Rde.
\end{equation*}
Then, for all $s>0$,
\begin{equation*}
  (F_c)^*(s)= F^*(c^d\,s).
\end{equation*}
\end{lem}


Also we need some other basic properties of the distributional
function and of the nonincreasing rearrangement of a function
(cf. \cite[Proposition~3.1]{LaMi23}).

\begin{lem}\label{SumDistrF}
Assume that $M_j\, (j\in\Ne)$, are measurable subsets of $\Rde$ which are pairwise disjoint.
Let $\{\psi_j\}_{j=1}^\infty$ be a sequence of measurable functions such that
$\spt\psi_j\subset M_j$, $j\in\Ne$. Then
for each $\tau>0$,
\begin{equation*}
  \big({\textstyle\sum_{j=1}^{\infty}\psi_j}\big)_{*}(\tau)
  =
  \sum_{j=1}^{\infty}\big(\psi_j\big)_{*}(\tau).
\end{equation*}
\end{lem}

\begin{proof}
The identity follows directly from \eqref{DistF}; we leave the details to the reader.
\end{proof}

\begin{lem}\label{SumNonIncrRear}
Assume that $M_j$, $j\in\Ne$, are measurable subsets of $\Rde$ which are pairwise disjoint.
Let $\{\psi_j\}_{j=1}^\infty$ be a sequence of measurable functions such that
$\spt\psi_j\subset M_j$, $j\in\Ne$, and let $I_j$, $j\in\Ne$, be  pairwise disjoint intervals in $(0,\infty)$.
Then,
for each $t>0$,
\begin{equation*}
  \big({\textstyle\sum_{j=1}^{\infty}\psi_j}\big)^{*}(t)
  \ge
  \sum_{j=1}^{\infty}\chi_{I_j}(t)(\psi_j)^{*}(t).
\end{equation*}
\end{lem}

\begin{proof}
The inequality is trivial for $t\in(0,\infty)\setminus\bigcup_{j\in\Ne}I_j$.
Assume that there is $t\in\bigcup_{j\in\Ne}I_j$ such that
\begin{equation}\label{assump}
  \big({\textstyle\sum_{j=1}^{\infty}\psi_j}\big)^{*}(t)
  <
  \sum_{j=1}^{\infty}\chi_{I_j}(t)(\psi_j)^{*}(t).
\end{equation}
Then there is a unique index $\ell\in\Ne$ such that $t\in I_\ell$. By \eqref{NonRear} there is a number
$\tau_0>0$ such that
\begin{equation*}
  \mu\big\{x\in\Rde\sT\big|{\textstyle\sum_{j=1}^{\infty}\psi_j(x)}\big|>\tau_0\}\le t
  \quad
  \text{and}
  \quad
  \tau_0<(\psi_\ell)^{*}(t).
\end{equation*}
On the other hand, \eqref{NonRear} and Lemma \ref{SumDistrF}   imply that
\begin{equation*}
  \mu\{x\in\Rde\sT |\psi_\ell(x)|>\tau_0\}>t
\end{equation*}
and we get a contradiction, and the assertion is proved.
\end{proof}

\section{Main result} \label{Sect3}

At first we recall a well known result.

\begin{thm} \label{BddFT}
The Fourier transform $\FT$ is a bounded linear
operator from $L^p(\Rde)$ into $L^{p',p}(\Rde)$,
that is,
 \begin{equation}\label{spojFTdoLor}
   \FT: L^p(\Rde)\to L^{p',p}(\Rde)
 \end{equation}
provided that $p\in(1,2)$.
\end{thm}


\begin{rem}
Theorem \ref{BddFT} above is proved using only the
Marcinkiewicz interpolation theorem (see, for example,
\cite[Theorem 4.13]{BS}) and the fundamental properties of the Fourier transform, namely:
\begin{equation*}
  \FT: L^1(\Rde) \to L^\infty(\Rde)\quad\text{and}\quad \FT: L^2(\Rde) \to L^2(\Rde).
\end{equation*}
Note that,
if we use the Riesz-Thorin interpolation theorem, we obtain
 boundedness
 \begin{equation*}
   \FT: L^p(\Rde)\to L^{p'}(\Rde)
 \end{equation*}
if $1\le p\le2$ (see e.g. \cite[Theorem 1.2.1]{MR482275}). For $p\in(1,2)$ this
result is weaker than \eqref{spojFTdoLor},
since $L^{p',p}(\Rde)$ is continuously embedded into $L^{p'}(\Rde)$.
\end{rem}

It follows from \cite[Theorem 5.1]{MR3167479} that
the Fourier transform
$\FT: L^p(\Rde)\to L^{p'}(\Rde)$
is finitely strictly singular if and only if  $p\in(1,2)$.

One of the main results of this note complements the above observation:

\begin{thm} \label{MainThhm}
Let $p\in(1,2)$. The Fourier transform
\begin{equation*}
  \FT: L^p(\Rde)\to L^{p',p}(\Rde)
\end{equation*}
is not strictly singular.
\end{thm}

\noindent
Note that this result contrasts with that of \cite{MR3167479}, in which it is shown
that the Fourier transform on the scale of Lebesgue spaces is finitely strictly singular for $1<p<2$.

\section{Proof of Theorem \ref{MainThhm}} \label{Sect4}

We begin this section with some elementary observations concerning the behaviour of the Fourier transform
under rescaling of functions (see Lemma~\ref{FToffa}). We then introduce the tent functions defined in \eqref{BasicTF},
and proceed to construct an infinite-dimensional system
\eqref{SystOFtestFu}--\eqref{FinalTestF02} that forms the core of our proof.

For the sake of simplicity, we will provide a detailed proof
of the theorem only for the one-dimensional case, that is,
for $d=1$. As the proof
is almost identical in the general case (where $d>1$),
we will only give a brief outline for it.
We start with introducing suitable \emph{test functions}.
For a fixed $a>0$ we
put
\begin{equation}\label{Funct_fa}
  f_a(t)=\chi_{[-1/a,1/a]}(t),\quad t\in\Re.
\end{equation}
It is easy to see that
\begin{equation*}
  \FT f_a(x)=\frac{2\sin(x/a)}{x},\quad x\in\Re.
\end{equation*}

\begin{lem}\label{FToffa}
Let $a>0$, $p>1$, and let $f_a$ be the function defined by \eqref{Funct_fa}.
Denote
\begin{equation*}
  c_p=2^{-1/p}\|\FT f_1\|_{p',p}.
\end{equation*}
Then
\begin{equation}\label{NormsOFfa}
  \|f_a\|_p=(2/a)^{1/p},\quad
  \|\FT f_a\|_{p',p}=c_p(2/a)^{1/p}.
\end{equation}
\end{lem}

\begin{proof}
The first equality is obvious.
We prove the second one.
Observe that
\begin{equation*}
  f_a(t)=f_1(at),\quad t\in\Re.
\end{equation*}
Using the linearity, scaling property of the Fourier transform and Lemma \ref{DilatL} we obtain
\begin{equation*}
  (\FT f_a)^*(\tau)
  =
  a^{-1}\big(\FT(f_1(x/a))\big)^*(\tau)
  =
  a^{-1}(\FT f_1)^*(\tau/a),\quad \tau>0.
\end{equation*}
Thus, applying this identity and making the change of variables $y=a\tau$ we arrive at
\begin{equation*}
   \|\FT f_a\|_{p',p}^p
   =
   \int_{0}^{\infty}\big(\tau^{1/p'}a^{-1}\,(\FT f_1)^*(\tau/a)\big)^p\frac{\dd \tau}{\tau}
   =
   a^{-1}\int_{0}^{\infty}\big(y^{1/p'}(\FT f_1)^*(y)\big)^p\frac{\dd y}{y}
   =
   a^{-1}\|\FT f_1\|_{p',p}^p,
\end{equation*}
that is, $\|\FT f_a\|_{p',p}=c_p(2/a)^{1/p}$.
\end{proof}

\paragraph{\bf Basic test function}\quad
Set
\begin{equation} \label{BasicTF}
  g_a(t)=(2/a)^{-1/p}f_a(t),\quad t\in\Re,
\end{equation}
where $a>0$ is a suitable constant.
Then, by \eqref{NormsOFfa},
\begin{equation}\label{NormsOFga}
  \|g_a\|_p=1,\quad
  \|\FT g_a\|_{p',p}=c_p.
\end{equation}
Given $\gamma>0$, we find, due to the absolute continuity of
norms $\|\cdot\|_p$ and $\|\cdot\|_{p',p}$, numbers $\eta\in(0,1)$,
$\nu^L$ and $\nu^R$, $0<\nu^L<\nu^R<\infty,$ such that
\begin{equation} \label{EstwithEP}
  \|g_1\,\chi_{[-1,-\eta)\cup(\eta,1]}\|_p\ge1-\gamma
\end{equation}
and
\begin{equation} \label{EstwithEP00}
  \|(\FT g_1)\,\chi_{(-\infty,-\nu^R)\cup[-\nu^L,\nu^L]\cup(\nu^R,\infty)}\|_{p',p}\le c_p\gamma/2,
\end{equation}
 that is,
\begin{equation} \label{EstwithEP02}
  \|(\FT g_1)\,\chi_{[-\nu^R,-\nu^L)\cup(\nu^L,\nu^R]}\|_{p',p}\ge c_p(1-\gamma/2).
\end{equation}
We have (see \eqref{LorqN})
\begin{equation*}
  \|(\FT g_1)\,\chi_{[-\nu^R,-\nu^L)\cup(\nu^L,\nu^R]}\|_{p',p}
  =
  \Big(\int_{0}^{2(\nu^R-\nu^L)}\big(y^{1/p'}(\FT g_1\,\chi_{[-\nu^R,-\nu^L)\cup(\nu^L,\nu^R]})^*(y)\big)^p\frac{\dd y}{y}\Big)^{1/p}.
\end{equation*}
Using the absolute continuity of the integral and \eqref{EstwithEP02}
we find numbers
$\dl^L$ and $\dl^R$, $0<\dl^L<\dl^R<2(\nu^R-\nu^L)$, such that
\begin{equation} \label{EstwithEP03}
  \Big(\int_{\dl^L}^{\dl^R}\big(y^{1/p'}(\FT g_1\,\chi_{[-\nu^R,-\nu^L)\cup(\nu^L,\nu^R]})^*(y)\big)^p\frac{\dd y}{y}\Big)^{1/p}
  \ge c_p(1-\gamma).
\end{equation}
By the same scaling argument which was used in the proof of Lemma~\ref{FToffa}, we obtain, for any $a>0$,
\begin{align}\label{ScaledTestF01}
 & \|g_a\,\chi_{[-1/a,-\eta/a)\cup(\eta/a,1/a]}\|_p\ge1-\gamma, \\
 \label{ScaledTestF02}
 &  \Big(\int_{\dl^La}^{\dl^Ra}\big(y^{1/p'}(\FT g_a\,
 \chi_{[-\nu^Ra,-\nu^La)\cup(\nu^La,\nu^Ra]})^*(y)\big)^p\frac{\dd y}{y}\Big)^{1/p}
 \ge c_p(1-\gamma).
\end{align}

\paragraph{\bf Construction of an infinite-dimensional subspace}\quad
Assume that $\ep>0$ is a fixed number.
For $\gamma=\ep/2$ we find numbers
and $\eta$, $\nu^L$ and $\nu^R$
such that the inequalities \eqref{EstwithEP} and \eqref{EstwithEP03} hold. Thus, inequalities \eqref{ScaledTestF01} and \eqref{ScaledTestF02}
hold with $\gamma=\ep/2^1$,  $\eta=\eta_1$,  $\nu^L=\nu^L_1$, $\nu^R=\nu^R_1$,
$\dl^L=\dl^L_1$, $\dl^R=\dl^R_1$,
and $a=a_1=1$.
In the next step we repeat this construction with $\gamma=\ep/2^2$ and find numbers $\eta_2\in(0,\eta_1)$,
$\nu^L_2\in(0,\nu^L_1)$ and $\nu^R_2\in(\nu^R_1,\infty)$ such that inequalities \eqref{EstwithEP},
\eqref{EstwithEP00} and \eqref{EstwithEP02} are satisfied
with $\gamma=\ep/2^2$,  $\eta=\eta_2$,  $\nu^L=\nu^L_2$, $\nu^R=\nu^R_2$, and
the selection
\begin{equation*}
  a_{2}>a_1\max\{1/\eta_1,\nu^R_1/\nu^L_2,\dl^R_1/\dl^L_2\}
\end{equation*}
results that
\begin{gather*}
  \big([-1/a_1,-\eta_1/a_1)\cup(\eta_1/a_1,1/a_1)\big)
  \cap
  \big([-1/a_2,-\eta_2/a_2)\cup(\eta_2/a_2,1/a_2)\big)
  =
  \emptyset,
\\
  \big([-\nu^R_1 a_1,-\nu^L_1 a_1)\cup(\nu^L_1 a_1,\nu^R_1 a_1]\big)
  \cap
  \big([-\nu^R_2 a_2,-\nu^L_2 a_2)\cup(\nu^L_2 a_2,\nu^R_2 a_2]\big)
  =
  \emptyset,
\end{gather*}
and
\begin{equation*}
  (a_1\dl^L_1,a_1\dl^R_1)\cap(a_2\dl^L_2,a_2\dl^R_2)
  =
  \emptyset.
\end{equation*}
We can continue this process and obtain sequences $\{\eta_j\}_{j=1}^\infty$,
$\big\{\nu^L_j\big\}_{j=1}^\infty$, $\big\{\nu^R_j\big\}_{j=1}^\infty$
$\big\{\dl^L_j\big\}_{j=1}^\infty$, $\big\{\dl^R_j\big\}_{j=1}^\infty$,
$\{a_j\}_{j=1}^\infty$
and the corresponding sequence of functions
\begin{equation} \label{SystOFtestFu}
  \f_j=g_{a_j},\quad j\in\Ne,
\end{equation}
 such that
\begin{gather}\label{FinalTestF01}
  \|\f_j\,\chi_{G_j}\|_p\ge\,1-\ep2^{-j}, \\
   \label{FinalTestLoF02}
  \Big(\int_0^\infty\big(y^{1/p'}(\FT \f_j\,
 \chi_{\Re\setminus\widehat{G}_j})^*(y)\big)^p\frac{\dd y}{y}\Big)^{1/p}\le\, c_p\ep2^{-j-1},
 \intertext{and}
 \label{FinalTestF02}
  \Big(\int_{I_j}\big(y^{1/p'}(\FT \f_j\,
 \chi_{\widehat{G}_j})^*(y)\big)^p\frac{\dd y}{y}\Big)^{1/p}\ge\, c_p(1-\ep2^{-j}),
\end{gather}
where the sets
\begin{equation*}
  G_j={[-1/a_j,-\eta_j/a_j)\cup(\eta_j/a_j,1/a_j]},\;
  \widehat{G}_j={[-\nu^R_j a_j,-\nu^L_j a_j)\cup(\nu^L_j a_j,\nu^R_j a_j]},\;
  I_j=(\dl^L_ja_j,\dl^R_ja_j)
\end{equation*}
satisfy
\begin{equation}\label{SystDisj}
  G_j\cap G_k=\emptyset\quad\text{and}\quad \widehat{G}_j\cap \widehat{G}_k=\emptyset
  \quad\text{and}\quad
  I_j\cap I_k=\emptyset
  \quad\text{if}\quad
  j,k\in\Ne, \;
  j\ne k.
\end{equation}

\begin{proof}[{\bf Proof of Theorem~\ref{MainThhm}}]
{\it Case $d=1$.}\quad
According to Definition~\ref{DefStrSing} we show that there is
an infinite dimensional subspace $X$ of $L^p(\Re)$ and a number $b>0$ such that, for any
$f\in X$, $\|f\|_p=1$, we have $\|\FT f\|_{p',p}\ge b$.

Given $\ep>0$. Consider the infinite-dimensional closed subspace
$X=\operatorname{span}\big(\{\f_j\}_{j=1}^\infty\big)$ of $L^p(\Re)$,
where the sequence
$\{\f_j\}_{j=1}^\infty$ is defined in \eqref{SystOFtestFu}.
If $f\in X$, then there are numbers $\al_j\in\Re$, $j\in\Ne$, and $A\in(0,\infty)$ such that
\begin{equation}\label{Funct_f}
  f=\sum_{j=1}^{\infty}\al_j\f_j\quad\text{and}\quad A^p=\sum_{j=1}^{\infty}|\al_j|^p.
\end{equation}
We estimate $\|f\|_p$ from above and $\|\FT f\|_{p',p}$ from below.
We start with the upper estimate of $\|f\|_p$. Denote
\begin{equation*}
  E_j=[-\eta/a_j,\eta/a_j]=[-1/a_j,1/a_j]\setminus G_j,\quad j\in\Ne.
\end{equation*}
Then obviously
\begin{equation*}
  \|\f_j\chi_{E_j}\|_p\le\ep2^{-j},\quad j\in\Ne.
\end{equation*}
By the Minkowski inequality we have
\begin{multline*}
  \|f\|_p
  =
  \Big\|\sum_{j=1}^{\infty}\al_j\f_j(\chi_{G_j}+\chi_{E_j})\Big\|_p
  =
  \Big\|\Big(\sum_{j=1}^{\infty}\al_j\f_j\chi_{G_j}\Big)+\Big(\sum_{j=1}^{\infty}\al_j\f_j\chi_{E_j}\Big)\Big\|_p \\
  \le
  \Big\|\sum_{j=1}^{\infty}\al_j\f_j\chi_{G_j}\Big\|_p
  +
  \Big\|\sum_{j=1}^{\infty}\al_j\f_j\chi_{E_j}\Big\|_p.
\end{multline*}
As concerns the first term on the right hand side, we use the fact that the functions have
disjoint supports (cf.~\eqref{SystDisj}) to obtain the estimate
\begin{equation*}
  \Big\|\sum_{j=1}^{\infty}\al_j\f_j\chi_{G_j}\Big\|_p^p
  =
  \sum_{j=1}^{\infty}|\al_j|^p\|\f_j\chi_{G_j}\|_p^p
  \le
  \sum_{j=1}^{\infty}|\al_j|^p\|\f_j\|_p^p
  =
  \sum_{j=1}^{\infty}|\al_j|^p=A^p.
\end{equation*}
The second term we estimate again by the Minkowski inequality
and the obvious inequality $|\al_j|\le A$
to arrive at
\begin{equation*}
   \Big\|\sum_{j=1}^{\infty}\al_j\f_j\chi_{E_j}\Big\|_p
   \le
   \sum_{j=1}^{\infty}|\al_j|\,\|\f_j\chi_{E_j}\|_p
   \le
   A\sum_{j=1}^{\infty}\|\f_j\chi_{E_j}\|_p
   \le
   A\sum_{j=1}^{\infty}\frac{\ep}{2^{j}}=A\ep.
\end{equation*}
Consequently,
\begin{equation} \label{UpperEst}
  \|f\|_p\le A(1+\ep).
\end{equation}
It remains to derive a lower estimate of the norm in $L^{p',p}(\Re)$
of the Fourier transform of $f$.
We have
\begin{multline}\label{LoeEsN1}
  \|\FT f\|_{p',p}
  =
  \Big\|\sum_{j=1}^{\infty}\al_j\FT\f_j\Big\|_{p',p}
  =
  \Big\|\sum_{j=1}^{\infty}\al_j
  (\chi_{\widehat{G}_j}+\chi_{\Re\setminus\widehat{G}_j})\FT\f_j\Big\|_{p',p} \\
  \ge
  \Big\|\sum_{j=1}^{\infty}\al_j\chi_{\widehat{G}_j}\FT\f_j\Big\|_{p',p}
  -
  \Big\|\sum_{j=1}^{\infty}\al_j\chi_{\Re\setminus\widehat{G}_j}\FT\f_j\Big\|_{p',p}.
\end{multline}
Note that $\|\cdot\|_{p',p}$ is a norm since $1<p<2<p'<\infty$.

We estimate the second term on the right hand side from above. Using \eqref{FinalTestLoF02}
we obtain
\begin{equation*}
  \Big\|\sum_{j=1}^{\infty}\al_j\chi_{\Re\setminus\widehat{G}_j}\FT\f_j\Big\|_{p',p}
  \le
  \sum_{j=1}^{\infty}|\al_j|\,\big\|\chi_{\Re\setminus\widehat{G}_j}\FT\f_j\big\|_{p',p}
  \le
  Ac_p\ep\sum_{j=1}^{\infty}2^{-j-1}=\tfrac12Ac_p\ep.
\end{equation*}
Together with \eqref{LoeEsN1} it implies that
\begin{equation} \label{LoEs01}
  \|\FT f\|_{p',p}
  \ge
  \Big\|\sum_{j=1}^{\infty}\al_j\chi_{\widehat{G}_j}\FT\f_j\Big\|_{p',p}-\tfrac12Ac_p\ep.
\end{equation}
The functions $\chi_{\widehat{G}_j}\FT\f_i$,  have disjoint supports $\widehat{G}_j$, $j\in\Ne$,
and so, by
Lemma~\ref{SumNonIncrRear}, we have
\begin{equation}\label{LoEstRearr}
  \big({\textstyle\sum_{j=1}^{\infty}\al_j\chi_{\widehat{G}_j}\FT\f_j}\big)^{*}(t)
  \ge
  \sum_{j=1}^{\infty}|\al_j|\chi_{I_j}(t)\big(\chi_{\widehat{G}_j}\FT\f_j\big)^{*}(t)
  \quad\text{ for all $t>0$}.
\end{equation}
Now we are ready to derive the lower estimate. Using
\eqref{LoEstRearr},
\eqref{SystDisj} and \eqref{FinalTestF02} we have
\begin{multline*}
  \Big\|\sum_{j=1}^{\infty}\al_j\chi_{\widehat{G}_j}\FT\f_j\Big\|_{p',p}^p
  =
  \int_{0}^\infty\big(y^{1/p'}
  \big({\textstyle\sum_{j=1}^{\infty}\al_j\chi_{\widehat{G}_j}\FT\f_j}\big)^{*}(y)\big)^p\frac{\dd y}{y}  \\
  \ge
  \sum_{j=1}^{\infty}|\al_j|^p
  \int_{I_j}\big(y^{1/p'}(\FT \f_j\,
 \chi_{\widehat{G}_j})^*(y)\big)^p\frac{\dd y}{y}
 \ge
 \sum_{j=1}^{\infty}|\al_j|^p(c_p(1-\ep2^{-j}))^p,
\end{multline*}
that is,
\begin{equation*}
  \Big\|\sum_{j=1}^{\infty}\al_j\chi_{\widehat{G}_j}\FT\f_j\Big\|_{p',p}
  \ge
  \Big(\sum_{j=1}^{\infty}|\al_j|^p\Big)^{1/p}c_p(1-\ep)
  =
  Ac_p(1-\ep).
\end{equation*}
Consequently, using \eqref{LoEs01}, we get
\begin{equation} \label{LowerEst}
  \|\FT f\|_{p',p}
  \ge
  Ac_p(1-\tfrac32\ep),
\end{equation}
by which we conclude the proof for $d=1$.

{\it Case $d>1$.}\quad
In this case we consider the cube $Q_a=[-1/a,1/a]^d$,
define the test function
\begin{equation*}\label{d_dimTest Fu}
  g_a(t)=(2/a)^{-d/p}\chi_{Q_a}(t),\quad t\in\Rde,
\end{equation*}
and observe that
\begin{equation*}\label{NormsOFga_dDIM}
  \|g_a\|_p=1\quad\text{and}\quad
  \|\FT g_a\|_{p',p}=c_{d,p},
\end{equation*}
where $c_{d,p}=\|\FT g_1\|_{p',p}=c_{d,p}$.
The proof is analogous to that of the one-dimensional case.
We leave the details to the reader.
\end{proof}

\section{The discrete case} \label{Sect5}

In this section we turn our attention to a discrete case, that is, to the
\emph{discrete Fourier transform}
\begin{equation*}
  \widehat{\boldsymbol\xi}_m=
  (\FT\boldsymbol\xi)(m)=
  \int_{\mathbb{T}^d}\boldsymbol\xi(t)e^{-i\langle t,m\rangle}\dd t,\quad
  \text{where $\boldsymbol\xi$ is an integrable function on $\mathbb{T}^d$,
  $m\in\mathbb{Z}^d$},
\end{equation*}
(cf.~\cite[Chapter~3]{Graf_ClFA14})
to be a mapping from the space $L^p(\mathbb{T}^d)$, $p\in(1,2)$,
on the $d$-dimensional torus $\mathbb{T}^d=[-\pi,\pi]^d$
into the sequence space $l^{p'}(\mathbb{Z}^d)$.
Recall that this means that there is a positive constant $C>0$ such that
\begin{equation*}
  \|\FT{\boldsymbol\xi}\|_{p'}
  =
  \Big(\sum_{m\in\mathbb{Z}^d}|\widehat{\boldsymbol\xi}_m|^{p'}\Big)^{1/p'}
  \le
  C\,\|\boldsymbol\xi\|_p .
\end{equation*}
It was shown by Weis in~\cite{MR635212} that the Fourier transform
\begin{equation}\label{DiscFT}
  \FT: L^p(\mathbb{T}^d)\to l^{p'}(\mathbb{Z}^d), \quad p\in(1,2),
\end{equation}
is strictly singular.
Later,  it was proved in \cite{MR3167479} that it is even finitely
strictly singular.

We recall that, for $p\in(1,\infty)$,
the \emph{Lorentz sequence space} $l^{p',p}(\mathbb{Z}^d)$ is the
space of all sequences $\mathbf{c}=\{c_m\}_{m\in\mathbb{Z}^d}$ such
that
\begin{equation*}
    \|\mathbf{c}\|_{p',p}=\Big(\sum_{k=1}^\infty\big(k^{1/p'-1/p}c_k^*)^p\Big)^{1/p}<\infty,
\end{equation*}
where $\{c_k^*\}_{k=1}^\infty$ denotes the non-increasing rearrangement of the sequence $\mathbf{c}$
(for more details see e.g.~\cite[page~51]{Ko}, or \cite[Section 1.18.3]{MR1328645}, \cite{MR4648846}).

\begin{thm} \label{MainDiscreteThm}
Let $p\in(1,2)$. The discrete Fourier transform
is a bounded linear operator from the space
$L^p(\mathbb{T}^d)$ into the sequence space $l^{p',p}(\mathbb{Z}^d)$,
however it is not strictly singular.
\end{thm}

Before we  prove the result we have to do some preliminary work, since the discrete case is slightly
different from the continuous one.
Concerning the test functions we use the same system as in the proof of Theorem~\ref{MainThhm}, that is,
as a basic test function we use the function $g_a$, where $a\in\Ne$ is sufficiently large, defined in~\eqref{BasicTF}.
We restrict to the one dimensional case, as the general case ($d>1$) can be treated similarly.
Assume that $a\in\Ne$ is a sufficiently large fixed number which we specify later.
We use the notation
\begin{equation*}
  (\FT g_a)_m=
  \int_{-\pi}^{\pi}g_a(t)e^{-itm}\dd t,\quad
   m\in\mathbb{Z},
\end{equation*}
in fact,
\begin{equation*}
  (\FT g_a)_m
  =
  \bigg\{
    \begin{array}{ll} \vspace{4pt}
      2(a/2)^{1/p}\sin(m/a)/{m} & \hbox{if $m\ne0$}, \\
      (2/a)^{1/p'} & \hbox{if $m=0$}.
    \end{array}
\end{equation*}

The most important relationship between the continuous and discrete cases is formulated in the following lemma.
\begin{lem} \label{LemDiscrEst}
Given $\gamma>0$. Then there exists $a_0\in\Ne$ such that, for any $a\ge a_0$,
\begin{equation}\label{MainDiscEst}
  \left|\,\|\FT g_a\|_{p',p}-2\|\{(\FT g_a)_m\}_{m\in\mathbb{Z}}\|_{p',p}\,\right|\le\gamma.
\end{equation}
\end{lem}

\begin{proof}
We link a step function $\widehat{g_a}^\text{\phantom{.}cont}$ to the function  $\FT g_a$.
This function is defined, for $x\in[0,\infty)$, by
\begin{equation*}
  \widehat{g_a}^\text{\phantom{.}cont}(x)=(\FT g_a)_m\quad\text{if $m\ge0$ and $x\in[m,m+1)$},
\end{equation*}
and, for $x\in(-\infty,0)$, by
\begin{equation*}
  \widehat{g_a}^\text{\phantom{.}cont}(x)=\widehat{g_a}^\text{\phantom{.}cont}(-x).
\end{equation*}
Define the reverse scaling transformation
\begin{equation*}
  T_a(f)(x)=(2a)^{-1/p}f(x/a)\quad x\in\Re.
\end{equation*}
Then
\begin{equation*}
  T_a(\FT g_a)=\FT g_1
\end{equation*}
and so (cf. Lemma~\ref{FToffa})
\begin{equation*}
  \|\FT g_a\|_{p',p}
  =
  \|T_a(\FT g_a)\|_{p',p}
  = c_p.
\end{equation*}
Since the intervals on which the function $T_a(\widehat{g_a}^\text{\phantom{.}cont})$
is constant can be made arbitrarily small by choosing a sufficiently large value for $a$, we obtain
\begin{equation*}
  |\FT g_1(x)-T_a(\widehat{g_a}^\text{\phantom{.}cont})(x)|<\gamma/2
  \quad\text{for all $x\in\Re$}.
\end{equation*}
Together with
\begin{equation*}
  \|\widehat{g_a}^\text{\phantom{.}cont}\|_{p',p}
  =
  \|T_a(\widehat{g_a}^\text{\phantom{.}cont})\|_{p',p},
\end{equation*}
it implies that
\begin{equation}  \label{Estimate01}
  \left|\,\|\FT g_a\|_{p',p}-\|\widehat{g_a}^\text{\phantom{.}cont}\|_{p',p}\,\right|\le\gamma/2.
\end{equation}
Observe that
\begin{equation*}
  \|T_a(\widehat{g_a}^\text{\phantom{.}cont})\|_{p',p}
  =
  \Big(\int_{0}^{\infty}\left(t^{1/p'-1/p}(T_a(\widehat{g_a}^\text{\phantom{.}cont}))^*(t)\right)^p\dd t\Big)^{1/p},
\end{equation*}
where
\begin{equation*}
  (T_a(\widehat{g_a}^\text{\phantom{.}cont}))^*(t)
  =\sum_{k=1}^{\infty}(\FT g_1)^*_k\,\chi_{[2(k-1)/a,2k/a)}(t),\quad t>0.
\end{equation*}
Since
the length of each of the intervals is $2/a$ (note that the function
$\widehat{g_a}^\text{\phantom{.}cont}$ is even) converges
to zero when $a$ goes to infinity, we can modify  the function
$t\mapsto t^{1/p'-1/p}$ to a step function being constant on the corresponding intervals
and approximating the original function. Taking $a$ sufficiently large we obtain the estimate
\begin{equation*}
  \left|\,
  \|T_a(\widehat{g_a}^\text{\phantom{.}cont})\|_{p',p}
  -\Big(\int_{0}^{\infty}\left(\psi_a(t)^{1/p'-1/p}
  (T_a(\widehat{g_a}^\text{\phantom{.}cont}))^*(t)\right)^p\dd t\Big)^{1/p}
  \,\right|\le\gamma/2,
\end{equation*}
where
\begin{equation*}
  \psi_a(t)=\sum_{k=1}^{\infty}k\chi_{[2(k-1)/a,2k/a)}(t),\quad t>0.
\end{equation*}
It is easy to verify that
\begin{equation*}
  \Big(\int_{0}^{\infty}\left(\psi_a(t)^{1/p'-1/p}
  (T_a(\widehat{g_a}^\text{\phantom{.}cont}))^*(t)\right)^p\dd t\Big)^{1/p}
  =
  2\|\{(\FT g_a)_m\}_{m\in\mathbb{Z}}\|_{p',p},
\end{equation*}
and so
\begin{equation*}
  \left|\,
  \|\widehat{g_a}^\text{\phantom{.}cont}\|_{p',p}
  -2\|\{(\FT g_a)_m\}_{m\in\mathbb{Z}}\|_{p',p}
  \,\right|\le\gamma/2.
\end{equation*}
Together with \eqref{Estimate01} it completes the proof of \eqref{MainDiscEst}.
\end{proof}

\begin{proof}[{\bf Sketch of the proof of Theorem \ref{MainDiscreteThm}}]
The boundedness of the discrete Fourier transform can be showed  using techniques as in Theorem~\ref{BddFT}.
The argument that the embedding is not strictly singular follows from the ideas presented
in the proof of Theorem~\ref{MainThhm}, suitably adapted to the setting of sequence spaces in a straightforward manner.
Thus, it suffices to highlight the main distinctions between the continuous and discrete cases.
	
	We use the same family of test functions as in the proof of Theorem~\ref{MainThhm}.
Specifically, we take the functions $g_a$ defined in~\eqref{BasicTF} as our basic test functions,
for which $\|g_a\|_p = 1$ and $a \in \mathbb{N}$. Given an arbitrary $\ep > 0$, we can construct a
sequence $\{a_j\}_{j=1}^\infty \subset \mathbb{N}$, with $a_1 \geq a_0$, where $a_0$ is the constant from
Lemma~\ref{LemDiscrEst}. This sequence ensures that a discrete version of the function $f$, defined in~\eqref{Funct_f},
satisfies the discrete version of the upper estimate~\eqref{UpperEst}.
	
	For the lower estimate, which is analogous to~\eqref{LowerEst}, Lemma~\ref{LemDiscrEst} guarantees
that the discrete case can be handled in a manner similar to the continuous one.
We omit the detailed calculations, as they closely follow the continuous case analysis.
\end{proof}

\section{Final Comments} \label{Sect6}

At the end of our contribution, we note that the techniques used in the proof of Theorem~\ref{MainThhm} (see Section~\ref{Sect4}),
with some rather technical modifications, can also be applied to establish the following result:

\begin{thm} \label{MainExtension}
Let $p\in(1, 2)$ and $q\in (1, \infty)$. The Fourier transform $\FT: L^{p,q}(\mathbb{R}^d) \to L^{p',q}(\mathbb{R}^d)$
and the discrete Fourier transform $\FT: L^{p,q}(\mathbb{T}^d) \to \ell^{p',q}(\mathbb{Z}^d)$
are not strictly singular.
\end{thm}

We conclude this final section with open questions arising from the results of this paper:
\begin{itemize}
  \item
Is the mapping $\FT: L^p(\Rde)\to  L^{p',r}(\Rde)$ for $p<r<p'$, finitely strictly singular?
  \item
Is there a rearrangement invariant space $Y(\Rde)$ such that
$\FT: L^p(\Rde)\to Y(\Rde)$ is strictly singular
but not finitely strictly singular for $1<p\le2$?
\end{itemize}


\end{document}